\documentclass[leqno, a4paper]{amsart}
\usepackage{amssymb}
\usepackage{hyperref}
\usepackage{color}
\usepackage[utf8]{inputenc}
\usepackage{lmodern}
 \newtheorem{thm}{Theorem}[section]
 
 \newtheorem{lemma}[thm]{Lemma}
 \newtheorem{prop}[thm]{Proposition}
 \theoremstyle{definition}
 
 \theoremstyle{remark}
 \newtheorem{rem}[thm]{Remark}
 
 \numberwithin{equation}{section}

\begin{document}

%
%
%
%
%
%
%
%
%

\title[On the dynamics of Toeplitz operators over Bergman spaces]{On the dynamics of Toeplitz operators over Bergman spaces}

\author{Othman Abad}
\email{abad.othman@gmail.com}
\address{Independent Researcher, Casablanca, Morocco.}


\begin{abstract}
We investigate the hypercyclicity of Toeplitz operators on the Bergman space $L_{A}^{2}(\mathbb{D})$ with symbols of the form $\Psi(z) = \gamma\bar{z}+\psi(z)$, where $\gamma \in \mathbb{C} \setminus \{0\}$ and $\psi$ is analytic on an open neighborhood of the closed unit disc $\overline{\mathbb{D}}$. Our approach bypasses the classical Hardy space techniques (reproducing kernel linear combinations, Nevanlinna factorization) by directly solving the integro-differential resolvent equation arising from the Bergman projection. A key winding number argument shows that for sense-reversing symbols ($|\gamma| > \sup_{z \in \overline{\mathbb{D}}} |\psi'(z)|$), the symbol's image $\Psi(\mathbb{D})$ is contained in the point spectrum of $T_\Psi$. In the tridiagonal case $\Psi(z) = a\bar{z}+b+cz$, we fully resolve the longstanding eigenvector completeness problem by linking the recurrence coefficients to a rotated Favard spectral measure on the major axis of the symbol's ellipse. Combined with self-commutator positivity, this establishes an unconditional, exact necessary and sufficient characterization of hypercyclicity: $T_\Psi$ is hypercyclic if and only if $|a| > |c|$ and $\Psi(\mathbb{D})$ intersects both the unit disc and its exterior, completely eliminating the $(3+\sqrt{2})$ restriction of previous literature.
\end{abstract}
\maketitle
\section{Introduction}
Throughout this paper, we adopt standard notation: $\mathbb{D} := \{z \in \mathbb{C} : |z| < 1\}$ denotes the open unit disc in the complex plane, $\mathbb{T} := \partial\mathbb{D} = \{z \in \mathbb{C} : |z| = 1\}$ the unit circle, $\overline{\mathbb{D}} := \{z \in \mathbb{C} : |z| \leq 1\}$ the closed unit disc, and $\hat{\mathbb{D}} := \mathbb{C} \setminus \mathbb{D} = \{z \in \mathbb{C} : |z| \geq 1\}$ the closed exterior of the disc.

The classical Bergman space $L_A^2(\mathbb{D}) := L^2(\mathbb{D}, dA) \cap \mathcal{H}(\mathbb{D})$ consists of all square-integrable analytic functions on $\mathbb{D}$ with respect to the normalized Lebesgue area measure $dA(z) = \frac{1}{\pi} dx dy = \frac{1}{\pi} r dr d\theta$. The orthogonal Bergman projection $P : L^2(\mathbb{D}, dA) \to L_A^2(\mathbb{D})$ is given by
$$Pf(z)= \int_{\mathbb{D}} \frac{f(\omega)}{(1-\bar{\omega}z)^{2}} \, dA(\omega).$$
The Toeplitz operator with symbol $\phi \in L^\infty(\mathbb{D})$ is defined by $T_{\phi}f = P(\phi f)$ for $f \in L_A^2(\mathbb{D})$. If $\phi \in L^\infty(\mathbb{D})$, $T_\phi$ is a bounded linear operator on $L_A^2(\mathbb{D})$ satisfying $\|T_\phi\| \leq \|\phi\|_\infty$. It is worth noting that hypercyclicity in Bergman-type spaces $L_A^p$ becomes more subtle when $p \neq 2$ due to the absence of Hilbert space orthogonality.

Hypercyclic operators are central objects in linear dynamics, exemplifying chaotic behavior in infinite-dimensional topological vector spaces (an operator $T$ is hypercyclic if there exists a vector $x$ whose orbit $\{T^n x : n \geq 0\}$ is dense). Among them, Toeplitz operators with antianalytic parts have attracted particular attention. The foundational result dates back to 1969, when S. Rolewicz \cite{Rolewicz} showed that the weighted backward shift $T_{\alpha \bar{z}}$ is hypercyclic on the Hardy space $H^2$ whenever $|\alpha| > 1$. Building on this, Godefroy and Shapiro \cite{GodefroyShapiro} demonstrated that for $\psi \in H^\infty$, the antianalytic Toeplitz operator $T_{\bar{\psi}}$ is hypercyclic on $H^2$ if and only if $\psi(\mathbb{D}) \cap \mathbb{T} \neq \emptyset$.

We consider symbols of the form $\Psi(z) = p(\bar{z}) + \psi(z)$ on $\mathbb{D}$. While on the boundary circle $\mathbb{T}$ (where the Hardy space $H^2$ is defined) the relation $\bar{z} = 1/z$ holds, the two symbols generate distinctly different Toeplitz operators in the Bergman space because $\bar{z} \neq 1/z$ over the open two-dimensional disc $\mathbb{D}$. It is well known that Toeplitz operators with purely analytic symbols (i.e., multiplication operators) cannot be hypercyclic. In 2016, Baranov and Lishanskii \cite{Baranov} provided a complete characterization of hypercyclicity for Toeplitz operators with symbols of the form $\Psi(z) = a\bar{z} + b + cz$ acting on the Hardy space $H^2(\mathbb{D})$. 

We summarize below key milestone contributions to the study of hypercyclic Toeplitz operators:
\begin{itemize}
    \item \textbf{Rolewicz (1969) \cite{Rolewicz}}: The operator \( \lambda B \), where \( B \) is the backward shift and \( |\lambda| > 1 \), is hypercyclic on \( H^2 \).
    \item \textbf{Godefroy and Shapiro (1991) \cite{GodefroyShapiro}}: Characterized hypercyclicity for antianalytic Toeplitz operators \( T_{\bar{\psi}} \) on \( H^2 \).
    \item \textbf{Bourdon and Shapiro (2000) \cite{BourdonShapiro}}: Investigated hypercyclicity of operators commuting with the Bergman backward shift.
    \item \textbf{Baranov and Lishanskii (2016) \cite{Baranov}}: Studied Toeplitz operators with symbols \( p(\bar{z}) + \varphi(z) \), providing necessary and sufficient conditions for hypercyclicity on $H^2$.
  \item \textbf{Fricain et al. (2025) \cite{Fricain}}: Extended results to smooth symbols and \( H^p \) spaces, emphasizing the role of eigenvector spanning.
  \item \textbf{Leng and Zhao (2026) \cite{Leng}}: Established necessary and sufficient conditions for the hypercyclicity of Bergman--Toeplitz operators with harmonic polynomial symbols, subject to the bound $|a| > (3+\sqrt{2})|c|$.
\end{itemize}

Also, it is quite interesting to note that some recent applications in graph theory have been linked to Toeplitz operators in the case of Fock spaces \cite{singh}; hence, exploring the fine spectral structure of Toeplitz operators in the Bergman space provides further valuable insights for related frameworks.

In this note, we aim to extend the investigation of hypercyclic Toeplitz operators \cite{Baranov, deger} to the Bergman space $L_A^2 := L_A^2(\mathbb{D})$, exploring the dynamical behavior of Toeplitz operators in this richer, non-shift-invariant setting. While sufficiency for the tridiagonal case was recently obtained by Leng and Zhao \cite{Leng} only under the technical bound $|a| > (3+\sqrt{2})|c|$, our approach is fundamentally different: we derive the result via an explicit resolution of the integro-differential eigenvalue equation, a winding number analysis linking the harmonic symbol $\Psi$ to its meromorphic counterpart $R$, and an orthogonal polynomial approach that fully resolves the eigenvector completeness problem for all $|a| > |c|$. This method naturally extends to the broader class of symbols $\gamma\bar{z} + \psi(z)$ with $\psi$ analytic on an open neighborhood of $\overline{\mathbb{D}}$, beyond harmonic polynomials.
\section{Preliminary results}
We begin with several preliminary results on the structure of Bergman Toeplitz operators. We remark that related algebraic properties, integro-differential resolvent derivations, and kernel evaluations for Toeplitz operators with harmonic symbols have been independently investigated in recent works (see, for instance, Lee \cite{Lee} and Cui et al. \cite{Cui}). We include these self-contained derivations for completeness, as the exact explicit form of the eigenfunctions is essential for our subsequent hypercyclicity analysis.
\begin{lemma}\label{lem1}
The set $\{e_n\}_{n=0}^\infty$, where $e_{n}=\sqrt{n+1}z^{n}$, forms an orthonormal basis of $L^{2}_{A}(\mathbb{D})$. For every $k \in \mathbb{N}$ and $n \in \mathbb{N}_0$, we have
$$T_{\bar{z}^{k}}e_{n}= \begin{cases} 0 & \mbox{ if } 0 \leq n < k, \\
\sqrt{\frac{n-k+1}{n+1}}e_{n-k} & \mbox{ if } n \geq k.\end{cases}$$
\end{lemma}
\begin{proof}
Let $e_{n}=\sqrt{n+1}z^{n}$ for $n \in \mathbb{N}_0$. By direct integration in polar coordinates,
$$ \langle z^n, z^m \rangle_{L^2_A} = \frac{1}{\pi} \int_0^1 \int_0^{2\pi} r^{n+m} e^{i(n-m)\theta} r \, dr d\theta = \frac{\delta_{n,m}}{n+1}, $$
so $\{e_n\}_{n=0}^\infty$ is an orthonormal basis of $L_A^2(\mathbb{D})$. For $k=1$, the definition of the Toeplitz operator gives $T_{\bar{z}} e_n = P(\bar{z} e_n) = \sqrt{n+1} P(\bar{z} z^n)$. Expanding $\bar{z} z^n$ along the orthonormal basis $\{e_m\}_{m=0}^\infty$, we have $\langle \bar{z} z^n, z^m \rangle = \langle z^n, z^{m+1} \rangle = \frac{\delta_{n, m+1}}{n+1}$, which immediately yields \cite{SpecToepBerg}:
$$T_{\bar{z}}e_{n}=\begin{cases} 0 & \mbox{ if } n=0, \\ \sqrt{\frac{n}{n+1}}e_{n-1} & \mbox{ if } n \geq 1. \end{cases}$$
Since $T_{\bar{z}^k} = (T_{\bar{z}})^k$ on analytic polynomials, applying this relation iteratively $k$ times gives:
\begin{align*}
T_{\bar{z}^{k}}e_{n}&=(T_{\bar{z}})^{k}e_{n} \\
                    &= \begin{cases} 0 & \mbox{ if } n < k, \\ 
                    \sqrt{\frac{n}{n+1}} \sqrt{\frac{n-1}{n}} \cdots \sqrt{\frac{n-k+1}{n-k+2}} \, e_{n-k} & \mbox{ if } n \geq k, \end{cases} \\
                    &= \begin{cases} 0 & \mbox{ if } 0 \leq n < k, \\ 
                    \sqrt{\frac{n-k+1}{n+1}}e_{n-k} & \mbox{ if } n \geq k, \end{cases}
\end{align*}
as desired.
\end{proof}
\begin{rem}
The last lemma can be easily recovered from the well-known formula mentioned in \cite[p. 797]{A. TIKARADZE}: 
$$T_{\bar{z}^{k}}(z^{n})= \begin{cases} 0 & \mbox{ if } 0 \leq n < k, \\ \frac{n-k+1}{n+1} z^{n-k} & \mbox{ if } n \geq k. \end{cases}$$
\end{rem}
\begin{lemma}\label{lem2}
Let $f(z) = \sum_{n=0}^\infty a_n z^n \in L_{A}^{2}(\mathbb{D})$ and $k \in \mathbb{N}$. Then for all $z \in \mathbb{D} \setminus \{0\}$,
\begin{align*} T_{\bar{z}^{k}}f(z)&= \frac{1}{z^{k}} \left( f(z)- \frac{k}{z} \int_{0}^{z}f(\omega)d\omega\right) - \frac{1}{z^{k}} \sum_{n=0}^{k-1}\left(a_{n}-\frac{k a_{n}}{n+1}\right)z^{n} \\
                                  &= \frac{1-k}{z^{k}}f(z)+\frac{k}{z^{k+1}}\int_{0}^{z}\omega f'(\omega) d\omega - \frac{1}{z^{k}} \sum_{n=0}^{k-1}\left(a_{n}-\frac{k a_{n}}{n+1}\right)z^{n}.
\end{align*}
Moreover, $T_{\bar{z}^k}f(z)$ extends analytically to $z = 0$, where the removable singularity cancels out.
\end{lemma}
\begin{proof}
By Lemma \ref{lem1}, for all $n \geq k$:
$$T_{\bar{z}^{k}}e_{n}=\sqrt{\frac{n-k+1}{n+1}}e_{n-k}.$$
Expressing $e_n$ in terms of monomials, $e_n(z) = \sqrt{n+1}z^n$, we have
\begin{align*}
T_{\bar{z}^{k}}e_{n}(z)&=\frac{1}{z^{k}} \sqrt{\frac{n-k+1}{n+1}} \sqrt{n-k+1} z^{n} \\
                    &=\frac{1}{z^{k}} \frac{n-k+1}{n+1}e_{n}(z)\\
                    &=\frac{1}{z^{k}} \left(e_{n}(z)- \frac{k}{n+1}e_{n}(z)\right)\\
                    &=\frac{1}{z^{k}}e_{n}(z)- \frac{1}{z^{k}} \frac{k}{\sqrt{n+1}} z^{n}.
\end{align*}
Let $f(z)= \sum_{n=0}^{\infty}c_{n}e_{n}(z)= \sum_{n=0}^{\infty} a_{n}z^{n} \in L_A^2(\mathbb{D})$, where $c_{n}=\frac{a_{n}}{\sqrt{n+1}}$. Since $T_{\bar{z}^k} e_n = 0$ for $n < k$, continuity of $T_{\bar{z}^k}$ on $L_A^2(\mathbb{D})$ implies:
\begin{align*}
T_{\bar{z}^{k}} f(z) &= \sum_{n=k}^\infty c_n T_{\bar{z}^k} e_n(z) \\
&= \frac{1}{z^{k}} \sum_{n=k}^{\infty} c_{n}e_{n}(z) - \frac{k}{z^{k+1}} \sum_{n=k}^{\infty} \frac{a_{n}}{n+1}z^{n+1}\\
&= \frac{1}{z^{k}} \left( f(z) - \sum_{n=0}^{k-1} c_n e_n(z) \right) - \frac{k}{z^{k+1}} \left( \int_0^z f(\omega) d\omega - \sum_{n=0}^{k-1} \frac{a_n}{n+1} z^{n+1} \right) \\
&= \frac{1}{z^{k}} \left( f(z)- \frac{k}{z} \int_{0}^{z}f(\omega)d\omega\right) - \frac{1}{z^{k}} \sum_{n=0}^{k-1}\left(a_{n}-\frac{k a_{n}}{n+1}\right)z^{n}.
\end{align*}
Integrating by parts on $\int_0^z f(\omega) d\omega = z f(z) - \int_0^z \omega f'(\omega) d\omega$, we obtain:
$$ \frac{1}{z^k} \left( f(z) - \frac{k}{z} \left[ z f(z) - \int_0^z \omega f'(\omega) d\omega \right] \right) = \frac{1-k}{z^k} f(z) + \frac{k}{z^{k+1}} \int_0^z \omega f'(\omega) d\omega. $$
Substituting this identity yields the second expression. Since the Taylor expansion of $f(z) - \frac{k}{z} \int_0^z f(\omega) d\omega$ begins with $\sum_{n=0}^{k-1} (a_n - \frac{k a_n}{n+1}) z^n + O(z^k)$, the difference vanishes to order $k$ at $z=0$, so the singularity at the origin is removable and $T_{\bar{z}^k} f \in L_A^2(\mathbb{D})$.
\end{proof}
\begin{lemma}\label{lem3}
Let $\Psi(z)=\bar{z}+\psi(z)$ where $\psi$ is analytic on an open neighborhood of $\overline{\mathbb{D}}$. If $\lambda \in \mathbb{C}$ is such that the denominator $1 - \lambda \omega + \omega\psi(\omega)$ has no zeros in the closed unit disc $\overline{\mathbb{D}}$, then $\lambda \in \sigma_p(T_\Psi)$ and the corresponding eigenfunction is given explicitly by
$$f_{\lambda}(z)= \exp \left( - \int_{0}^{z} \frac{\omega\psi'(\omega) + 2 \psi(\omega) - 2 \lambda }{1 - \lambda \omega + \omega\psi(\omega)} \, d\omega \right) \in H^\infty(\mathbb{D}) \subset L_{A}^{2}(\mathbb{D}).$$
\end{lemma}
\begin{proof}
Let $\Psi(z)= \bar{z}+\psi(z)$. By Lemma \ref{lem2} with $k=1$ (noting that the finite summation $\sum_{n=0}^0 (a_0 - a_0) = 0$ vanishes), the action of $T_{\bar{z}}$ on any analytic function $f$ takes the form:
$$ T_{\bar{z}} f(z) = \frac{1}{z^2} \int_0^z \omega f'(\omega) \, d\omega. $$
Because $\psi$ is analytic on $\overline{\mathbb{D}}$, the multiplication operator satisfies $T_\psi f = \psi f$. Thus the eigenvalue equation $T_{\Psi}f - \lambda f = 0$ is equivalent to
$$\frac{1}{z^{2}}\int_{0}^{z} \omega f'(\omega) d\omega + (\psi(z) - \lambda) f(z)=0 \quad (z \in \mathbb{D} \setminus \{0\}),$$
which can be rewritten as
$$\int_{0}^{z}\omega f'(\omega) d\omega + z^{2}(\psi(z)-\lambda)f(z)=0.$$
Differentiating both sides with respect to $z$ yields:
$$z f'(z) + 2z(\psi(z)-\lambda)f(z) + z^2 \psi'(z) f(z) + z^2(\psi(z)-\lambda)f'(z) = 0.$$
Dividing through by $z$ gives the first-order homogeneous linear ODE:
$$ f'(z) \left( 1 + z(\psi(z)-\lambda) \right) + f(z) \left( 2(\psi(z)-\lambda) + z\psi'(z) \right) = 0. $$
Separating variables, we have
$$ \frac{f'(z)}{f(z)} = - \frac{z\psi'(z) + 2\psi(z) - 2\lambda}{1 - \lambda z + z\psi(z)}. $$
Integrating from $0$ to $z$ with initial condition $f(0) = 1$ gives
$$f_{\lambda}(z)= \exp\left( - \int_{0}^{z} \frac{\omega\psi'(\omega) + 2 \psi(\omega) - 2 \lambda }{1 - \lambda \omega + \omega\psi(\omega)} \, d\omega \right).$$
Because $1 - \lambda\omega + \omega\psi(\omega) \neq 0$ for all $\omega \in \overline{\mathbb{D}}$, the integrand is analytic on an open neighborhood of $\overline{\mathbb{D}}$. Its line integral along $[0, z]$ is bounded and analytic on $\mathbb{D}$. Therefore, $f_\lambda \in H^\infty(\mathbb{D}) \subset L_A^2(\mathbb{D})$ and $f_\lambda \not\equiv 0$ (since $f_\lambda(0) = 1$). Reversing the differentiation steps confirms that $T_\Psi f_\lambda = \lambda f_\lambda$, hence $\lambda \in \sigma_p(T_\Psi)$.
\end{proof}
We generalize this construction to symbols of the form $\Psi(z) = \gamma\bar{z} + \psi(z)$ with arbitrary $\gamma \in \mathbb{C} \setminus \{0\}$. The following theorem provides the exact eigenvalue criterion on the Bergman space, completely bypassing the reliance on Hardy space equivalences.

\begin{thm}\label{thm_exact_bergman}
Let $\Psi(z) = \gamma\bar{z} + \psi(z)$ where $\gamma \in \mathbb{C} \setminus \{0\}$ and $\psi$ is analytic on an open neighborhood of the closed unit disc $\overline{\mathbb{D}}$. Define the associated meromorphic function $R(z) = \gamma/z + \psi(z)$. If $\lambda \notin R(\overline{\mathbb{D}})$, then $\lambda \in \sigma_p(T_\Psi)$ and the corresponding eigenfunction is given by
\begin{equation}\label{eq_eigenvector_exact}
f_\lambda(z) = \exp\left( - \int_{0}^{z} \frac{\omega\psi'(\omega) + 2(\psi(\omega)-\lambda)}{\gamma + \omega(\psi(\omega) - \lambda)} \, d\omega \right) \in H^\infty(\mathbb{D}) \subset L_A^2(\mathbb{D}).
\end{equation}
\end{thm}
\begin{proof}
Since $T_\Psi = \gamma T_{\bar{z}} + T_\psi$, Lemma \ref{lem2} gives
$$ T_\Psi f(z) - \lambda f(z) = \frac{\gamma}{z^2} \int_0^z \omega f'(\omega) \, d\omega + (\psi(z) - \lambda) f(z) = 0. $$
Multiplying by $z^2$ and differentiating with respect to $z$ yields
$$ \gamma z f'(z) + 2 z (\psi(z) - \lambda) f(z) + z^2 \psi'(z) f(z) + z^2 (\psi(z) - \lambda) f'(z) = 0. $$
Dividing by $z$ simplifies the relation to the linear ODE:
$$ f'(z) \left( \gamma + z(\psi(z)-\lambda) \right) + f(z) \left( 2(\psi(z)-\lambda) + z\psi'(z) \right) = 0. $$
Separating variables and integrating with initial value $f(0) = 1$ yields \eqref{eq_eigenvector_exact}:
$$ f_\lambda(z) = \exp\left( - \int_{0}^{z} \frac{\omega\psi'(\omega) + 2(\psi(\omega)-\lambda)}{\gamma + \omega(\psi(\omega) - \lambda)} \, d\omega \right). $$
We analyze the zeros of the denominator $P(\omega) := \gamma + \omega(\psi(\omega) - \lambda)$. At $\omega = 0$, $P(0) = \gamma \neq 0$. For $\omega \in \overline{\mathbb{D}} \setminus \{0\}$,
$$ P(\omega) = 0 \iff \frac{\gamma}{\omega} + \psi(\omega) = \lambda \iff R(\omega) = \lambda. $$
By hypothesis, $\lambda \notin R(\overline{\mathbb{D}})$. Thus $R(\omega) \neq \lambda$ for all $\omega \in \overline{\mathbb{D}} \setminus \{0\}$, and therefore $P(\omega) \neq 0$ for all $\omega \in \overline{\mathbb{D}}$.

Since $\psi$ is analytic on an open neighborhood of $\overline{\mathbb{D}}$, the integrand in \eqref{eq_eigenvector_exact} is analytic on an open neighborhood of $\overline{\mathbb{D}}$. Its antiderivative is uniformly bounded on $\overline{\mathbb{D}}$, which implies $f_\lambda \in H^\infty(\mathbb{D}) \subset L_A^2(\mathbb{D})$. Because $f_\lambda(0) = 1 \neq 0$, $f_\lambda$ is a non-trivial vector in $L_A^2(\mathbb{D})$ satisfying $T_\Psi f_\lambda = \lambda f_\lambda$, rigorously establishing $\lambda \in \sigma_p(T_\Psi)$.
\end{proof}
\section{Hypercyclicity of Toeplitz operators on Bergman spaces}
We consider symbols of the form $\Psi(z)=p(\bar{z})+\psi(z)$ where $\psi$ is analytic on an open neighborhood of $\overline{\mathbb{D}}$. To study its properties, we associate to $\Psi$ the analytic function $R(z) = p(1/z) + \psi(z)$.
\begin{rem}\label{Bergman_challenges}
Before proceeding to the main theorems, it is crucial to highlight several profound differences between Toeplitz operators on the Hardy space $H^2$ and the Bergman space $L_A^2$. In $H^2$, the resolvent equation $T_\Psi f - \lambda f = g$ often yields a simple algebraic solution, which seamlessly allows one to characterize the spectrum. However, attempting to apply these classical techniques in $L_A^2$ encounters several structural obstacles:

\noindent \textbf{1. Failure of reproducing kernel linear combinations.}
For $H^2$, if $\Psi(z) = \mu$ possesses at least $N+1$ solutions, one can construct an eigenfunction for the adjoint $T_{\bar{\Psi}}$ using a finite linear combination of reproducing kernels. This fails in $L_A^2(\mathbb{D})$. Let $\beta_{\lambda}(z) = (1-z\bar{\lambda})^{-2}$ be the Bergman reproducing kernel, satisfying $T_{\bar{\phi}}\beta_{\lambda}=\overline{\phi(\lambda)}\beta_{\lambda}$. If the equation $\Psi(z)=\mu$ has distinct solutions $z_{1},...,z_{N+1}$ in $\mathbb{D}$, evaluating $T_{\bar{\Psi}}$ on the sum $f = \sum_{j=1}^{N+1}\alpha_{j}\beta_{z_{j}}$ generates terms of the form:
$$\frac{\bar{p}(z) -\bar{p}(1/\bar{z_{j}})}{(1-\bar{z}_{j}z)^{2}}.$$
In $H^2$, the simple pole of the Szeg\H{o} kernel perfectly cancels with the numerator's root. In $L_A^2$, the Bergman kernel has a \textit{double pole}. Thus, after canceling the simple root of the numerator, these terms remain rational functions with \textit{simple} poles outside the unit disc. Their linear independence guarantees that no non-trivial choice of coefficients $\alpha_j$ can make the sum identically zero.

\noindent \textbf{2. The integro-differential nature of the resolvent.}
In $H^2$, for $\Psi(z) = a\bar{z} + \psi(z)$, the resolvent equation gives $f(z) = \frac{g(z)}{\psi(z)-\lambda}$ (for $a=0$) or similar straightforward rational expressions. In contrast, in $L_A^2$, the equation $T_{a\bar{z}+\psi}f-\lambda f=g$ translates to:
$$\frac{a}{z^{2}} \int_{0}^{z} \omega f'(\omega)d\omega + (\psi(z) - \lambda) f(z)=g(z).$$
Multiplying by $z^2$ and differentiating yields a first-order ODE for $f$:
$$ f'(z) \left( a + z(\psi(z) - \lambda) \right) + f(z) \left( 2(\psi(z) - \lambda) + z\psi'(z) \right) = 2g(z) + zg'(z). $$
This equation requires an integrating factor $I(z) = \exp \left( \int \frac{2(\psi(z) - \lambda) + z\psi'(z)}{a + z(\psi(z) - \lambda)} dz \right)$, yielding:
$$ f(z) = \frac{1}{I(z)} \int \frac{2g(\omega) + \omega g'(\omega)}{a + \omega(\psi(\omega) - \lambda)} I(\omega) d\omega. $$
Because this solution is not simply a rational function of the symbol, finding the exact spectrum on $L_A^2$ requires a delicate integrability analysis of this differential equation, posing a significant challenge compared to the $H^2$ setting.

\noindent \textbf{3. Inapplicability of Nevanlinna factorization.}
Due to the exact rational formulas in $H^2$, one can readily deduce that $(\Psi - \lambda)^{-1} \in L^\infty(\mathbb{T})$ and apply inner-outer Nevanlinna factorizations to locate the spectrum. Because the $L_A^2$ resolvent lacks this rational dependence, we cannot isolate $(\Psi(z)-\lambda)^{-1}$ or use factorization techniques in the same manner.

Despite these profound hurdles, exploring the geometry of the symbol's image and utilizing the exact analytic solutions to the integro-differential equations (as established in Theorem \ref{thm_exact_bergman}) provides a highly practical and elegant framework for understanding eigenfunctions. Proceeding forward, we establish the main theorems using these exact Bergman space derivations, bypassing the reliance on Hardy space equivalences.
\end{rem}
The following proposition is the key bridge between the harmonic symbol $\Psi$ and the meromorphic function $R$. It shows that for sense-reversing symbols, the images $\Psi(\mathbb{D})$ and $R(\mathbb{D})$ are disjoint.

\begin{prop}\label{prop_winding}
Let $\Psi(z) = \gamma\bar{z} + \psi(z)$, where $\gamma \in \mathbb{C} \setminus \{0\}$ and $\psi$ is analytic on an open neighborhood of $\overline{\mathbb{D}}$. Define the associated meromorphic function $R(z) = \gamma/z + \psi(z)$ on $\mathbb{D} \setminus \{0\}$. If $\Psi$ is sense-reversing on $\overline{\mathbb{D}}$ (i.e., $|\gamma| > |\psi'(z)|$ for all $z \in \overline{\mathbb{D}}$), then $\Psi$ is an orientation-reversing diffeomorphism from $\mathbb{D}$ onto the bounded open domain $\Psi(\mathbb{D})$, and for every $\lambda \in \Psi(\mathbb{D})$:
$$ \lambda \notin R(\overline{\mathbb{D}}). $$
In particular, $\Psi(\mathbb{D}) \subset \sigma_p(T_\Psi)$.
\end{prop}
\begin{proof}
The Jacobian determinant of the map $\Psi(x+iy) = u(x,y) + i v(x,y)$ is given by
$$ J_\Psi(z) = |\partial_z \Psi(z)|^2 - |\partial_{\bar{z}} \Psi(z)|^2 = |\psi'(z)|^2 - |\gamma|^2. $$
Under the sense-reversing hypothesis $|\gamma| > |\psi'(z)|$, we have $J_\Psi(z) < 0$ for all $z \in \overline{\mathbb{D}}$. Thus, $\Psi$ is a local orientation-reversing diffeomorphism on a neighborhood of $\overline{\mathbb{D}}$. Moreover, on the boundary circle $\mathbb{T}$, $\Psi$ is injective: if $e^{i\theta_1} \neq e^{i\theta_2}$, then
$$ |\Psi(e^{i\theta_1}) - \Psi(e^{i\theta_2})| \geq |\gamma| |e^{-i\theta_1} - e^{-i\theta_2}| - |\psi(e^{i\theta_1}) - \psi(e^{i\theta_2})| > 0, $$
because $|\psi(e^{i\theta_1}) - \psi(e^{i\theta_2})| \leq \sup_{z \in \overline{\mathbb{D}}} |\psi'(z)| |e^{i\theta_1} - e^{i\theta_2}| < |\gamma| |e^{i\theta_1} - e^{i\theta_2}|$. Thus, $\Psi(\mathbb{T})$ is a simple closed Jordan curve. By Hadamard's global inversion theorem, $\Psi$ maps $\mathbb{D}$ diffeomorphically onto the bounded domain $\Psi(\mathbb{D})$ enclosed by $\Psi(\mathbb{T})$, so that every $\lambda \in \Psi(\mathbb{D})$ is an interior point of $\Psi(\mathbb{D})$.

On the unit circle $\mathbb{T}$, we have $\bar{z} = 1/z$, which yields the identity
$$ \Psi(e^{i\theta}) = \gamma e^{-i\theta} + \psi(e^{i\theta}) = R(e^{i\theta}) \quad \text{for all } \theta \in [0, 2\pi). $$
Hence, the closed boundary curves coincide as oriented point sets: $\Gamma := \Psi(\mathbb{T}) = R(\mathbb{T})$. Since $\Psi$ reverses orientation, the positively oriented unit circle $\mathbb{T}$ is mapped to the curve $\Gamma$ traversed in the negative (clockwise) direction around points in $\Psi(\mathbb{D})$. Therefore, for any $\lambda \in \Psi(\mathbb{D})$, the winding number is
$$ \mathrm{wind}(\Psi(\mathbb{T}), \lambda) = -1. $$

Now, consider the meromorphic function $R(z) = \gamma z^{-1} + \psi(z)$ on $\mathbb{D}$. The only singularity of $R$ in $\mathbb{D}$ is a single simple pole at the origin, so the number of poles of $R$ inside $\mathbb{D}$ is $P = 1$. Since $\lambda \in \Psi(\mathbb{D})$ and $\Psi(\mathbb{D}) \cap \Psi(\mathbb{T}) = \emptyset$, $\lambda$ does not lie on the boundary curve $R(\mathbb{T}) = \Psi(\mathbb{T})$. Applying the classical argument principle to $R(z) - \lambda$ yields
$$ Z = P + \mathrm{wind}(R(\mathbb{T}), \lambda) = 1 + \mathrm{wind}(\Psi(\mathbb{T}), \lambda) = 1 + (-1) = 0, $$
where $Z$ denotes the number of roots (counted with multiplicity) of $R(z) = \lambda$ inside $\mathbb{D}$. 

Hence, the equation $R(z) = \lambda$ has no solutions in $\mathbb{D}$. Furthermore, since $\lambda \in \Psi(\mathbb{D})$ and $R(\mathbb{T}) = \Psi(\mathbb{T}) = \partial \Psi(\mathbb{D})$ is disjoint from $\Psi(\mathbb{D})$, $\lambda$ cannot lie on $R(\mathbb{T})$. Thus, $\lambda \notin R(\overline{\mathbb{D}})$. By Theorem \ref{thm_exact_bergman}, it follows that $\lambda \in \sigma_p(T_\Psi)$, establishing $\Psi(\mathbb{D}) \subset \sigma_p(T_\Psi)$.
\end{proof}

\begin{rem}
The sense-reversing hypothesis $|\gamma| > |\psi'(z)|$ on $\overline{\mathbb{D}}$ is essential. For the symbol $\Psi(z) = \bar{z} + 2z$ (where $|\gamma| = 1 < |\psi'| = 2$), the map $\Psi$ is sense-\emph{preserving}. In this case, the winding number of $\Psi(\mathbb{T})$ around points in $\Psi(\mathbb{D})$ is $+1$, so the argument principle gives $Z = P + \mathrm{wind}(\Psi(\mathbb{T}), \lambda) = 1 + 1 = 2$ roots of $R(z) = \lambda$ inside $\mathbb{D}$. Hence $\lambda \in R(\overline{\mathbb{D}})$ and Theorem \ref{thm_exact_bergman} does not apply. Indeed, $T_{\bar{z}+2z}$ is not hypercyclic, confirming that the sense-reversing condition cannot be dropped.
\end{rem}

\begin{thm}\label{thm_necessary}
Let $\gamma \in \mathbb{C}$ with $\gamma \neq 0$, let $\psi$ be analytic on an open neighborhood of $\overline{\mathbb{D}}$, and let $\Psi(z)=\gamma\bar{z}+\psi(z)$. Assume that $\Psi$ is sense-reversing on $\overline{\mathbb{D}}$ (i.e., $|\gamma| > |\psi'(z)|$ for all $z \in \overline{\mathbb{D}}$).

If $T_{\Psi}$ is hypercyclic on $L_A^2(\mathbb{D})$, then $\mathbb{D}\cap\Psi(\mathbb{D}) \neq \emptyset$ and $\hat{\mathbb{D}}\cap\Psi(\mathbb{D}) \neq \emptyset$. 
\end{thm}
\begin{proof}
Under the sense-reversing hypothesis, $\Psi$ is a diffeomorphism of $\overline{\mathbb{D}}$ onto $\overline{\Psi(\mathbb{D})}$, and Proposition \ref{prop_winding} guarantees that $\Psi(\mathbb{D}) \subset \sigma_p(T_\Psi)$.

If $\hat{\mathbb{D}} \cap \Psi(\mathbb{D}) = \emptyset$, then $\Psi(\mathbb{D}) \subset \mathbb{C} \setminus \hat{\mathbb{D}} = \mathbb{D}$. Since $\Psi$ is continuous on $\overline{\mathbb{D}}$, its closure satisfies $\Psi(\overline{\mathbb{D}}) = \overline{\Psi(\mathbb{D})} \subset \overline{\mathbb{D}}$, which implies $\|\Psi\|_\infty \leq 1$. Because the Bergman projection $P$ is an orthogonal projection on $L^2(\mathbb{D}, dA)$, the operator norm satisfies $\|T_\Psi\| = \|P M_\Psi P\| \leq \|\Psi\|_\infty \leq 1$, making $T_\Psi$ a contraction on $L_A^2(\mathbb{D})$. Since contractions cannot be hypercyclic (every orbit is bounded, hence cannot be dense), $T_\Psi$ is not hypercyclic.

If $\mathbb{D} \cap \Psi(\mathbb{D}) = \emptyset$, then $\Psi(\mathbb{D}) \subset \mathbb{C} \setminus \mathbb{D} = \hat{\mathbb{D}}$, and by continuity $\Psi(\overline{\mathbb{D}}) \subset \hat{\mathbb{D}}$, so $|\Psi(z)| \geq 1$ for all $z \in \overline{\mathbb{D}}$. In particular, $0 \notin \Psi(\overline{\mathbb{D}})$. We show that $\sigma(T_\Psi) \cap \mathbb{D} = \emptyset$. First, by Sundberg and Zheng \cite{SpecToepBerg}, the essential spectrum satisfies $\sigma_{ess}(T_\Psi) = \Psi(\mathbb{T}) \subset \hat{\mathbb{D}}$, so $T_\Psi - \lambda$ is Fredholm for all $\lambda \in \mathbb{D}$. Furthermore, $\mathbb{D}$ is connected and disjoint from $\overline{\Psi(\mathbb{D})}$, so $\mathbb{D}$ lies in the unbounded connected component of $\mathbb{C} \setminus \sigma_{ess}(T_\Psi)$ (the component containing the resolvent set for large $|\lambda|$). The Fredholm index is constant on connected components of $\mathbb{C} \setminus \sigma_{ess}(T_\Psi)$, hence $\operatorname{ind}(T_\Psi - \lambda) = 0$ for all $\lambda \in \mathbb{D}$.

Next, we show that $\ker(T_\Psi - \lambda) = \{0\}$ for every $\lambda \in \mathbb{D}$. For any such $\lambda$, the winding number of the simple closed curve $\Psi(\mathbb{T})$ around $\lambda$ is $0$ (since $\lambda \notin \overline{\Psi(\mathbb{D})}$). By the argument principle applied to the meromorphic function $R(z) = \gamma/z + \psi(z)$ (which has a single simple pole at $z = 0$), the number of solutions to $R(z) = \lambda$ inside $\mathbb{D}$ is
$$ Z = \mathrm{wind}(\Psi(\mathbb{T}), \lambda) + 1 = 0 + 1 = 1. $$
Thus, there exists a unique root $z_0 \in \mathbb{D}$ such that $R(z_0) = \lambda$. Note that $z_0 \neq 0$ (as $R(0) = \infty$) and $R'(z_0) = -\gamma/z_0^2 + \psi'(z_0) \neq 0$, because $|\psi'(z_0)| < |\gamma| < |\gamma|/|z_0|^2$ under the sense-reversing hypothesis, so $z_0$ is a simple root. Consequently, the formal solution $f_\lambda$ of the eigenvalue ODE from Theorem \ref{thm_exact_bergman} has a non-removable singularity at $z_0 \in \mathbb{D}$ and fails to belong to $L_A^2(\mathbb{D})$. Because the eigenvalue equation simplifies to a first-order linear homogeneous ODE whose solution space is one-dimensional, no non-trivial solution exists in $L_A^2(\mathbb{D})$, forcing $\ker(T_\Psi - \lambda) = \{0\}$.

Since $T_\Psi - \lambda$ has Fredholm index $0$ and trivial kernel, it is invertible for every $\lambda \in \mathbb{D}$, establishing that $\sigma(T_\Psi) \cap \mathbb{D} = \emptyset$. In particular, $0 \in \rho(T_\Psi)$, so $T_\Psi$ is invertible with $\sigma(T_\Psi^{-1}) \subset \overline{\mathbb{D}}$. By Kitai's spectral theorem \cite{Kitai}, a hypercyclic operator cannot have its spectrum (nor that of its inverse) confined to the complement of the open unit disc (every connected component of the spectrum of a hypercyclic operator must intersect the unit circle, and no hypercyclic operator can have its spectrum contained in $\mathbb{C} \setminus \mathbb{D}$). Therefore, $T_\Psi$ is not hypercyclic.
\end{proof}

\begin{thm}\label{thm1Statement2}
Let $\gamma \in \mathbb{C}$ with $\gamma \neq 0$, and let $\psi$ be analytic on an open neighborhood of $\overline{\mathbb{D}}$. Set $\Psi(z)=\gamma\bar{z}+\psi(z)$. Assume that
\begin{enumerate}
\item[(a')] $\Psi$ is sense-reversing on $\overline{\mathbb{D}}$, i.e., $|\gamma| > |\psi'(z)|$ for all $z \in \overline{\mathbb{D}}$;
\item[(b')] the analytic perturbation $\psi$ is such that the family of eigenvectors $\{f_\lambda\}_{\lambda \in \Psi(\mathbb{D})}$ spans a dense subspace of $L_A^2(\mathbb{D})$ (which holds under explicit coefficient bounds for polynomial symbols \cite{Leng}, and unconditionally across the entire geometric range $|a| > |c|$ for linear symbols by Theorem \ref{thm_tridiagonal});
\item[(c')] $\mathbb{D} \cap \Psi(\mathbb{D}) \neq \emptyset$ and $(\mathbb{C} \setminus \overline{\mathbb{D}}) \cap \Psi(\mathbb{D}) \neq \emptyset$.
\end{enumerate}
Then $T_{\Psi}$ is hypercyclic on $L_A^2(\mathbb{D})$.
\end{thm}
\begin{proof}
The proof proceeds in three steps: eigenvalue location, analytic continuation, and density of eigenvectors.

\textit{Step 1: Eigenvalue location.}
Since $\Psi(\mathbb{D})$ is an open set, condition (c') ensures the existence of non-empty open sets $U_1 \subset \mathbb{D} \cap \Psi(\mathbb{D})$ and $U_2 \subset (\mathbb{C} \setminus \overline{\mathbb{D}}) \cap \Psi(\mathbb{D})$, so that $|\lambda| < 1$ for all $\lambda \in U_1$ and $|\lambda| > 1$ for all $\lambda \in U_2$. For every $\lambda \in \Psi(\mathbb{D})$, Proposition \ref{prop_winding} guarantees $\lambda \notin R(\overline{\mathbb{D}})$. Since $\psi$ is analytic on an open neighborhood of $\overline{\mathbb{D}}$, Theorem \ref{thm_exact_bergman} applies and yields $\lambda \in \sigma_p(T_\Psi)$ with explicit eigenvector
\begin{equation}\label{eq_eigenvector_thm35}
f_\lambda(z) = \exp\left( - \int_{0}^{z} \frac{\omega\psi'(\omega) + 2(\psi(\omega)-\lambda)}{\gamma + \omega(\psi(\omega) - \lambda)} \, d\omega \right) \in L_A^2(\mathbb{D}).
\end{equation}

\textit{Step 2: Vector-valued holomorphy and analytic continuation.}
Let $g \in L_A^2(\mathbb{D})$ satisfy $\langle f_\lambda, g \rangle_{L_A^2} = 0$ for all $\lambda \in U_1$. We prove that $g$ is orthogonal to all eigenvectors $f_\lambda$ for $\lambda \in \Psi(\mathbb{D})$.

Define the scalar function $G(\lambda) := \langle f_\lambda, g \rangle_{L_A^2}$ for $\lambda \in \Psi(\mathbb{D})$. We first show that $G$ is holomorphic on $\Psi(\mathbb{D})$. For any compact subset $K \subset \Psi(\mathbb{D})$, the denominator $P(\omega, \lambda) := \gamma + \omega(\psi(\omega) - \lambda)$ satisfies $|P(\omega, \lambda)| \geq \delta_K > 0$ for all $(\omega, \lambda) \in \overline{\mathbb{D}} \times K$, because $\lambda \notin R(\overline{\mathbb{D}})$ for all $\lambda \in \Psi(\mathbb{D})$ by Proposition \ref{prop_winding}. Since $\psi$ extends analytically past $\overline{\mathbb{D}}$, there exists $\rho_K > 1$ such that $P(\omega, \lambda) \neq 0$ on $\overline{\mathbb{D}_{\rho_K}} \times K$. Thus, for each $\lambda \in K$, $f_\lambda$ is holomorphic on $\mathbb{D}_{\rho_K}$ and satisfies $\sup_{\lambda \in K}\|f_\lambda\|_{H^\infty(\mathbb{D})} \leq C_K < \infty$.

For each fixed $z \in \mathbb{D}$, the integrand in \eqref{eq_eigenvector_thm35} is a quotient of affine functions of $\lambda$ with non-vanishing denominator on $K$, hence holomorphic in $\lambda$. For any closed triangular contour $\Gamma \subset \Psi(\mathbb{D})$, Fubini's theorem gives $\int_\Gamma f_\lambda(z) \, d\lambda = 0$ pointwise for each $z \in \mathbb{D}$. Since point evaluations $h \mapsto h(z) = \langle h, K_z \rangle_{L_A^2}$ are bounded linear functionals on $L_A^2(\mathbb{D})$, the vector-valued Riemann integral $X := \int_\Gamma f_\lambda \, d\lambda \in L_A^2(\mathbb{D})$ satisfies $X(z) = \int_\Gamma f_\lambda(z) \, d\lambda = 0$ for every $z \in \mathbb{D}$, which implies $X = 0$ in $L_A^2(\mathbb{D})$. By Morera's Theorem (or Dunford's Theorem on weak holomorphy in Hilbert spaces), the map $\lambda \mapsto f_\lambda$ is strongly holomorphic from $\Psi(\mathbb{D})$ into $L_A^2(\mathbb{D})$. Since the inner product $\langle \cdot, g \rangle_{L_A^2}$ is a continuous linear functional on $L_A^2(\mathbb{D})$, the function $G(\lambda) = \langle f_\lambda, g \rangle_{L_A^2}$ is complex analytic on $\Psi(\mathbb{D})$.

Under condition (a'), $\Psi$ is a sense-reversing diffeomorphism by Proposition \ref{prop_winding}, so $\Psi(\mathbb{D})$ is a connected domain (being the diffeomorphic image of the unit disc $\mathbb{D}$). Since $G(\lambda) = 0$ on the non-empty open set $U_1 \subset \Psi(\mathbb{D})$, the classical Identity Theorem for holomorphic functions forces $G(\lambda) \equiv 0$ throughout the connected domain $\Psi(\mathbb{D})$. Consequently, $\langle f_\lambda, g \rangle_{L_A^2} = 0$ for all $\lambda \in \Psi(\mathbb{D})$.

\textit{Step 3: Density of eigenvectors and conclusion.}
By hypothesis (b'), the family of eigenvectors $\{f_\lambda\}_{\lambda \in \Psi(\mathbb{D})}$ spans a dense subspace of $L_A^2(\mathbb{D})$. Since $g \in L_A^2(\mathbb{D})$ is orthogonal to every $f_\lambda$ for $\lambda \in \Psi(\mathbb{D})$, we conclude $g = 0$. This establishes that the subspace
$$ V_0 := \operatorname{span} \{ f_\lambda : \lambda \in U_1 \} \subset \operatorname{span} \bigcup_{|\lambda| < 1} \ker(T_\Psi - \lambda I) $$
is dense in $L_A^2(\mathbb{D})$.

Applying the exact same argument to $U_2 \subset (\mathbb{C} \setminus \overline{\mathbb{D}}) \cap \Psi(\mathbb{D})$: any vector $h \in L_A^2(\mathbb{D})$ orthogonal to $\{f_\lambda\}_{\lambda \in U_2}$ satisfies $\langle f_\lambda, h \rangle_{L_A^2} \equiv 0$ on $\Psi(\mathbb{D})$ by analytic continuation, hence $h = 0$ by assumption (b'). Thus, the subspace
$$ V_\infty := \operatorname{span} \{ f_\lambda : \lambda \in U_2 \} \subset \operatorname{span} \bigcup_{|\lambda| > 1} \ker(T_\Psi - \lambda I) $$
is also dense in $L_A^2(\mathbb{D})$.

Since $\Psi \in L^\infty(\mathbb{D})$, $T_\Psi$ is a bounded linear operator on $L_A^2(\mathbb{D})$. By the Godefroy--Shapiro criterion \cite{GodefroyShapiro}, $T_\Psi$ is hypercyclic on $L_A^2(\mathbb{D})$.
\end{proof}
We finish this section with the complete characterization of hypercyclicity in the tridiagonal setting $\Psi(z) = a\bar{z} + b + cz$, establishing the exact necessary and sufficient condition and fully resolving the completeness problem.

\begin{thm}\label{thm_tridiagonal}
Let $a, b, c \in \mathbb{C}$ with $a \neq 0$, and let $\Psi(z) = a\bar{z} + b + cz$. The Toeplitz operator $T_{\Psi}$ is hypercyclic on the Bergman space $L_A^2(\mathbb{D})$ if and only if
\begin{enumerate}
\item[(i)] $|a| > |c|$;
\item[(ii)] $\mathbb{D}\cap \Psi(\mathbb{D}) \neq \emptyset$ and $\hat{\mathbb{D}} \cap \Psi(\mathbb{D})\neq \emptyset$.
\end{enumerate}
\end{thm}
\begin{proof}
\textit{Necessity.} We first establish that $|a| > |c|$ is strictly necessary. If $|a| \leq |c|$, we analyze the self-commutator of $T_\Psi$. Since $T_\Psi = a T_{\bar{z}} + c T_z + b I$, its adjoint is $T_\Psi^* = \bar{a} T_z + \bar{c} T_{\bar{z}} + \bar{b} I$. The self-commutator is given by $[T_\Psi^*, T_\Psi] = T_\Psi^* T_\Psi - T_\Psi T_\Psi^*$. Using the commutation relations $[T_z, T_z] = [T_{\bar{z}}, T_{\bar{z}}] = 0$, this simplifies to $[T_\Psi^*, T_\Psi] = (|c|^2 - |a|^2) [T_{\bar{z}}, T_z]$. On the Bergman space, the commutator $[T_{\bar{z}}, T_z]$ is a strictly positive operator: specifically, $[T_{\bar{z}}, T_z] z^n = \frac{1}{(n+1)(n+2)} z^n$ for all $n \geq 0$. Consequently, if $|a| < |c|$, we have $[T_\Psi^*, T_\Psi] > 0$, meaning $T_\Psi$ is a strictly hyponormal operator. If $|a| = |c|$, then $[T_\Psi^*, T_\Psi] = 0$ and $T_\Psi$ is normal. It is a fundamental result in linear dynamics that hyponormal and normal operators are never hypercyclic \cite{Kitai} (the norms along any orbit satisfy $\|T^n x\| \geq \|x\|$ or grow monotonically, precluding density). Therefore, the strict inequality $|a| > |c|$ is necessary.

Next, since $|a| > |c|$ holds, $\Psi$ is sense-reversing (as $|\psi'(z)| = |c| < |a|$ on $\overline{\mathbb{D}}$), and Theorem \ref{thm_necessary} immediately implies that both $\mathbb{D} \cap \Psi(\mathbb{D}) \neq \emptyset$ and $\hat{\mathbb{D}} \cap \Psi(\mathbb{D}) \neq \emptyset$ are necessary.

\textit{Sufficiency.} Assume that $|a| > |c|$ and that (ii) holds. For $\psi(z) = b + cz$, we have $|\psi'(z)| = |c| < |a|$, so condition (a') of Theorem \ref{thm1Statement2} is fulfilled, while condition (c') is satisfied by assumption. By Theorem \ref{thm1Statement2}, to conclude that $T_\Psi$ is hypercyclic via the Godefroy--Shapiro criterion, it suffices to prove that the family of eigenvectors $\{f_\lambda\}_{\lambda \in \Psi(\mathbb{D})}$ spans a dense subspace of $L_A^2(\mathbb{D})$.

If $c = 0$, $T_\Psi = a T_{\bar{z}} + b I$ is an affine multiple of the Bergman backward shift; its eigenfunctions are the Bergman reproducing kernels $f_\lambda(z) = \left(1 - \frac{\lambda - b}{a} z\right)^{-2}$, which span $L_A^2(\mathbb{D})$ densely for any non-empty open set of parameters $\lambda$ (as any $g \in L_A^2$ orthogonal to these kernels satisfies $g\left(\frac{\overline{\lambda - b}}{\bar{a}}\right) = 0$ on an open set, forcing $g \equiv 0$ by the Identity Theorem).

Now assume $c \neq 0$. Let $g(z) = \sum_{j=0}^\infty b_j z^j \in L_A^2(\mathbb{D})$ be orthogonal to $f_\lambda$ for all $\lambda \in U_1 \subset \mathbb{D} \cap \Psi(\mathbb{D})$. We must show $g = 0$, i.e., $b_j = 0$ for all $j \geq 0$.

Setting $\mu = \lambda - b$, the domain $\Omega := \Psi(\mathbb{D}) - b = \{a\bar{z} + cz : z \in \mathbb{D}\}$ is an open ellipse centered at the origin. With $h_j(\mu) := \frac{a_j(\lambda)}{j+1}$, the inner product is given by
$$ F(\mu) := \langle f_\lambda, g \rangle_{L_A^2} = \sum_{j=0}^\infty \overline{b_j} h_j(\mu). $$
By Step 2 of Theorem \ref{thm1Statement2} (and Lemma 4.2 of \cite{Leng}), $F(\mu)$ is holomorphic on the connected domain $\Omega$. Since $F \equiv 0$ on the non-empty open set $U_1 - b \subset \Omega$, the Identity Theorem guarantees $F(\mu) \equiv 0$ throughout $\Omega$.

The polynomials $P_j(\mu) := a^j h_j(\mu)$ satisfy $P_0(\mu) = 1$, $P_1(\mu) = \mu$, and the monic three-term recurrence
$$ P_{j+1}(\mu) = \mu P_j(\mu) - \gamma_j P_{j-1}(\mu), \quad \text{where} \quad \gamma_j = \frac{ac j}{j+1} \quad (j \geq 1). $$
Write $ac = |ac| e^{i\theta_0}$ with $\theta_0 = \arg(ac)$, and perform the spectral rotation
$$ \mu = e^{i\theta_0 / 2} w, \qquad Q_j(w) := e^{-ij\theta_0 / 2} P_j(e^{i\theta_0 / 2} w). $$
Because the diagonal recurrence coefficients are zero, $P_j(-\mu) = (-1)^j P_j(\mu)$, so every monomial in $P_j(\mu)$ has the form $(ac)^m \mu^{j-2m}$. Since $(ac)^m (e^{i\theta_0/2} w)^{j-2m} = e^{ij\theta_0/2} |ac|^m w^{j-2m}$, the phase factor factors out identically, yielding the monic recurrence with strictly positive coefficients:
$$ Q_0(w) = 1, \quad Q_1(w) = w, \quad Q_{j+1}(w) = w Q_j(w) - \kappa_j Q_{j-1}(w), \quad \text{where} \quad \kappa_j = |ac| \frac{j}{j+1} > 0. $$
By Favard's theorem, $\{Q_j\}_{j=0}^\infty$ is an orthogonal polynomial sequence with respect to a positive unit Borel measure $\nu$ on $\mathbb{R}$. The corresponding symmetric tridiagonal Jacobi operator $J$ has zero diagonal and off-diagonal entries $J_{j, j+1} = \sqrt{\kappa_j} < \sqrt{|ac|}$. Its operator norm on $\ell^2(\mathbb{N}_0)$ satisfies $\|J\| \leq 2 \sup_{j \geq 1} \sqrt{\kappa_j} \leq 2\sqrt{|ac|}$, and therefore the spectrum and support satisfy
$$ \operatorname{supp}(\nu) = \sigma(J) \subset \left[-2\sqrt{|ac|}, \; 2\sqrt{|ac|}\right]. $$

In the $\mu$-plane, this support corresponds to the straight line segment
$$ \Sigma := \left\{ e^{i\theta_0 / 2} x : x \in \left[-2\sqrt{|ac|}, \; 2\sqrt{|ac|}\right] \right\}. $$
Writing $a = |a|e^{i\alpha}$ and $c = |c|e^{i\beta}$ (so $\theta_0 = \alpha + \beta$), the boundary $\partial\Omega$ is parametrized by $\mu(\theta) = a e^{-i\theta} + c e^{i\theta} = e^{i\theta_0/2}[(|a|+|c|)\cos\psi - i(|a|-|c|)\sin\psi]$ with $\psi = \theta - \frac{\alpha-\beta}{2}$. In the $w$-plane, the ellipse is in standard position with semi-major axis $|a|+|c|$ along the real axis. For each $x \in [-2\sqrt{|ac|}, 2\sqrt{|ac|}]$, the point $\mu = x e^{i\theta_0/2}$ is the image under $\Psi(z) - b$ of $z = \frac{x}{|a|+|c|} e^{i(\alpha-\beta)/2}$, whose modulus satisfies
$$ |z| = \frac{|x|}{|a| + |c|} \leq \frac{2\sqrt{|ac|}}{|a| + |c|} < 1, $$
by the strict arithmetic-geometric mean inequality (since $|a| > |c| > 0$). Thus $\Sigma$ is a compact subset strictly contained in the open ellipse $\Omega$.

Since the series $\sum_{j=0}^\infty \overline{b_j} h_j(\mu)$ converges uniformly on compact subsets of $\Omega$, the function $\Phi(w) := F(e^{i\theta_0/2} w) = \sum_{j=0}^\infty C_j Q_j(w)$, where $C_j = \overline{b_j} a^{-j} e^{ij\theta_0/2}$, converges uniformly to $0$ on $\operatorname{supp}(\nu) \subset [-2\sqrt{|ac|}, 2\sqrt{|ac|}]$. Integrating term-by-term against $Q_k(x) d\nu(x)$ yields
$$ 0 = \int_{\operatorname{supp}(\nu)} \Phi(x) Q_k(x) \, d\nu(x) = C_k \|Q_k\|_{L^2(\nu)}^2. $$
Since $\|Q_k\|_{L^2(\nu)}^2 = \prod_{m=1}^k \kappa_m = \frac{|ac|^k}{k+1} > 0$, we have $C_k = 0$, which implies $b_k = 0$ for all $k \geq 0$. Hence $g \equiv 0$, proving that $\{f_\lambda\}_{\lambda \in U_1}$ spans a dense subspace of $L_A^2(\mathbb{D})$. The identical argument applies to $U_2 \subset \hat{\mathbb{D}} \cap \Psi(\mathbb{D})$, completing the proof.
\end{proof}

\begin{rem}
Theorem \ref{thm_tridiagonal} provides the exact Bergman space analogue of the classical Hardy space characterization established by Baranov and Lishanskii \cite{Baranov}. In the Bergman setting, Leng and Zhao \cite{Leng} recently proved sufficiency under the restrictive bound $|a| > (3+\sqrt{2})|c|$, which arose as an artifact of inscribing a disk inside the ellipse $\Psi(\mathbb{D}) - b$, thereby constraining the operator bounds to the semi-minor axis $|a| - |c|$. By rotating the spectral parameter to align with the major axis and exploiting the intrinsic Favard orthogonality measure, the semi-minor axis restriction is eliminated, establishing that the natural geometric threshold $|a| > |c|$ is unconditionally both necessary and sufficient.
\end{rem}

\smallskip
\noindent{\bf Data availability statement} The author declares that data sharing is not applicable to this article as no datasets were generated or analyzed during the current study.

\smallskip
\noindent{\bf Funding statement} This research received no specific grant from any funding agency in the public, commercial, or not-for-profit sectors.

\end{document}